\documentclass[12pt]{article}

\newif\ifpreprint   \preprinttrue
\newif\ifbiblatex   \biblatextrue

\usepackage{amsthm,amsmath,amssymb,mathtools}
\usepackage{graphicx,booktabs}
\usepackage[dvipsnames]{xcolor}
\usepackage{microtype}

\ifpreprint
  \usepackage[margin=1in]{geometry}
  \usepackage{authblk}
  \usepackage{newtxtext,newtxmath}   
\else
  \providecommand{\affil}[2][]{}     
\fi

\ifbiblatex
  \usepackage[style=alphabetic, maxbibnames=4]{biblatex}
  \let\citep\cite
  \let\citet\cite
\fi

\usepackage{my-macros}

\ifpreprint
  \usepackage[colorlinks=true,linkcolor=Maroon,
    urlcolor=MidnightBlue,citecolor=MidnightBlue]{hyperref}
\else
  \usepackage[colorlinks=false]{hyperref}
\fi
\usepackage[capitalise,noabbrev]{cleveref}

\providecommand{\keywords}[1]{\small\textbf{\textit{Keywords---}} #1}

\title{Distributional Shrinkage II: Higher-Order Scores \\Encode Brenier Map\thanks{Previously circulated as ``Distributional Shrinkage II: Optimal Transport Denoisers with Higher-Order Scores'' (\href{https://arxiv.org/abs/2512.09295v1}{arXiv:2512.09295v1}).}}

\author{Tengyuan Liang}
\affil{The University of Chicago}
\date{\today}

\begin{document}
\maketitle

\begin{abstract} 
Consider the additive Gaussian model $Y = X + \sigma Z$, where $X \sim P$ is an unknown signal, $Z \sim \mathcal{N}(0,1)$ is independent of $X$, and $\sigma > 0$ is known. Let $Q$ denote the law of $Y$. We construct a hierarchy of denoisers $T_0, T_1, \ldots, T_\infty \colon \mathbb{R} \to \mathbb{R}$ that depend only on higher-order score functions $q^{(m)}/q$, $m \geq 1$, of $Q$ and require no knowledge of the law $P$. The $K$-th order denoiser $T_K$ involves scores up to order $2K{-}1$ and satisfies $W_r(T_K \sharp Q, P) = O(\sigma^{2(K+1)})$ for every $r \geq 1$; in the limit, $T_\infty$ recovers the monotone optimal transport map (Brenier map) pushing $Q$ onto $P$.

We provide a complete characterization of the combinatorial structure governing this hierarchy through partial Bell polynomial recursions, making precise how higher-order score functions encode the Brenier map. We further establish rates of convergence for estimating these scores from $n$ i.i.d.\ draws from $Q$ under two complementary strategies: (i) plug-in kernel density estimation, and (ii) higher-order score matching. The construction reveals a precise interplay among higher-order Fisher-type information, optimal transport, and the combinatorics of integer partitions.
\end{abstract}






\section{Introduction}

Let $X \in \R$ be an unobserved signal drawn from an unknown distribution $P$, and let $Z$ be a standard Gaussian random variable independent of $X$. We observe only the noisy measurement $Y$ from an additive model, with noise level $\sigma \in \R_{>0}$ known a priori:
\begin{align*}
Y = X + \sigma Z, ~ Z \sim \mathcal{N}(0,1) \;.
\end{align*}
We denote the distribution of $Y$ by $Q$ and its density function with respect to the Lebesgue measure by $q: \R \rightarrow \R_{\geq 0}$.

\subsection{Stein's Paradox and Shrinkage}
Stein's paradox \citep{james1961estimation,brown1966admissibility,efron1973stein} states that it is possible to estimate $X$ from $Y$ with a smaller mean squared error (MSE) than $Y$ itself. Consider the estimator $Y + h(Y)$, where $h: \R \rightarrow \R$ is a denoising function. Stein's unbiased risk estimate (SURE) \citep{stein1981estimation} gives an unbiased estimate of the MSE:
\begin{align}
&\textrm{MSE}(h) := \E[\| Y + h(Y) - X \|^2] \nonumber \\
&=\sigma^2 + \E [h^2(Y)] + 2 \E [h(X +\sigma Z) \sigma Z] \nonumber \\
&=\sigma^2 + \E [h^2(Y)] + 2 \sigma^2 \E [h'(X+\sigma Z)]   \nonumber \\
& \qquad \left(\text{by integration by parts:}~ \textstyle\int z e^{-z^2/2}  h(x +\sigma z) \dd z = \sigma \int h'(x+\sigma z) e^{-z^2/2} \dd z\right) \nonumber \\
&= \sigma^2 + \E \left[ h^2(Y) + 2 \sigma^2 h'(Y) \right] \label{eq:SURE} \\
&= \sigma^2 + \int  q(y) \left( h^2(y) - 2 \sigma^2 \frac{q'(y)}{q(y)} h(y) \right) \dd y  \nonumber\\
& \qquad \left(\text{by integration by parts:}~ \textstyle\int h'(y) q(y) \dd y = - \int h(y) q'(y) \dd y\right) \nonumber\\
&= \sigma^2 - \sigma^4 \E  \left(\frac{q'(Y)}{q(Y)}\right)^2 < \sigma^2 \;,  \qquad \text{when}~ h(y) = \sigma^2 \frac{q'(y)}{q(y)} \nonumber \;.
\end{align}

Stein's paradox motivates the study of shrinkage denoisers that leverage the distributional information of $Q$. The Bayes-optimal denoiser is a prototypical example of such shrinkage denoisers, defined as
\begin{align*}
h(y) := \E[X|Y = y] - y =  \sigma^2 \frac{q'(y)}{q(y)} \;.
\end{align*}

The other notable example is the James-Stein estimator. Take $P = \frac{1}{d} \sum_{j=1}^d \delta_{X_j}$ as a discrete distribution with $d \geq 2$ atoms. For $j=1, \ldots, d$, let $Z_j$ be independent standard Gaussian noise and $Y_j = X_j + \sigma Z_j$. With the vectors denoted by $\bZ = (Z_1, \ldots, Z_d)$ and $\bY = (Y_1, \ldots, Y_d)$, the James-Stein estimator \citep{james1961estimation} is equivalent to the component-wise denoiser,
\begin{align*}
h_{\bY_{-j}}(Y_j) := -\sigma^2 \frac{(d-2) Y_j}{\|\bY\|^2} = -\sigma^2 \frac{(d-2) Y_j}{\|\bY_{-j}\|^2 + Y_j^2} \;, \quad j = 1, \ldots, d \;.
\end{align*}
Similarly, the derivation above applies:
\begin{align*}
\textrm{MSE}(h) &:= \frac{1}{d} \sum_{j=1}^d \E_{\bZ}\!\left[ |Y_j + h_{\bY_{-j}}(Y_j) - X_j|^2\right] \\
&= \sigma^2 + \frac{1}{d} \sum_{j=1}^d \E_{\bZ} \left[ h_{\bY_{-j}}^2(Y_j) + 2 \sigma^2 \partial_{Y_j} h_{\bY_{-j}}(Y_j) \right]  \quad \text{(same as \eqref{eq:SURE})}\\
&= \sigma^2 + \sigma^4 \E_{\bZ} \Bigg[ \frac{(d-2)^2}{d} \sum_{j=1}^d  \frac{Y_j^2}{\|\bY\|^4}  - \frac{2(d-2)}{d} \sum_{j=1}^d \left( \frac{1}{\|\bY\|^2} - \frac{2 Y_j^2}{\|\bY\|^4} \right) \Bigg]  \\
&= \sigma^2 - \sigma^4 \frac{(d-2)^2}{d} \E_{\bZ} \left[ \frac{1}{\| \bY\|^2} \right] < \sigma^2 \;, \qquad \text{when}~ d \geq 3 \;.
\end{align*}

As seen, both shrinkage-based denoisers require distributional information about $Q = \mathcal{L}(Y)$ or the empirical distribution $\widehat{Q} := \tfrac{1}{d} \sum_{j=1}^d \delta_{Y_j}$ to reduce the MSE.

A natural question is whether one can achieve better denoising performance than the Bayes-optimal denoiser above. The answer is negative under MSE, by definition. Can we denoise more effectively in a distributional sense, namely under a metric between the signal and the denoised distributions, say the Wasserstein metric? The answer is positive \citep{liang2025distributional}: one can denoise significantly better, with accuracy approaching zero. This is the focus of the current paper. At a high level, we study (i) representations of the optimal denoising map that induce a hierarchy of denoisers under the Wasserstein metric, (ii) how higher-order score functions play a fundamental role in constructing such denoisers, and (iii) how to estimate such denoisers from i.i.d. samples $\{Y_i\}_{i=1}^n$ drawn from $Q$.

Why the Wasserstein metric? In the case where $X \sim \frac{1}{d} \sum_{j=1}^d \delta_{X_j}$, a discrete distribution,
\begin{align*}
W_2^2\big( \mathcal{L}(Y + h(Y)), \mathcal{L}(X) \big) = \min_{\varsigma \in S_d} \frac{1}{d} \sum_{j=1}^d \E \left[ |Y_j + h(Y_j) - X_{\varsigma(j)}|^2\right] \nonumber \leq  \frac{1}{d} \sum_{j=1}^d  \E \left[ |Y_j + h(Y_j) - X_j|^2\right]
\end{align*}
where $\varsigma \in S_d$ is a permutation. More generally, for any random variables $X$ and $Y$,
\begin{align*}
	 W_2^2 \big( \mathcal{L}(Y + h(Y)), \mathcal{L}(X) \big) \leq \E \left| Y + h(Y) - X \right|^2 \;,
\end{align*}
where the inequality follows by considering better couplings beyond the naive coupling $(Y + h(Y), X)$.
Allowing for permutations, or more generally, couplings, offers leeway for higher denoising accuracy, but, perhaps more importantly, it also opens up new theoretical and practical perspectives.
In theory, it uncovers new mathematical structures from advanced combinatorics \citep{comtet1974advanced}, optimal transport \citep{villani2008optimal}, and higher-order score functions \citep{bobkov2024fisher}.
In practice, for several modern applications in generative modeling, such as image denoising via score-based diffusion models \citep{song2020score,sohl2015deep,liang2024denoising}, the Wasserstein metric is more relevant than the MSE; practitioners care about the quality of the denoiser at the distributional level, not just at the data level.

Traditional data-level shrinkage methods including Bayes-optimal and empirical Bayes denoisers \citep{efron1973stein,james1961estimation} often over-shrink the distribution. They produce denoised distributions that are overly concentrated and do not match the actual signal distribution \citep{garcia2024new,jaffe2025constrained,liang2025distributional}. We refer the reader to \citep{liang2025distributional,jaffe2025constrained} for examples and discussions of the benefits of distribution-level shrinkage compared to traditional data-level shrinkage. This paper continues along the path of \citep{garcia2024new,jaffe2025constrained,liang2025distributional} in studying classic signal denoising through the lens of optimal transport.

\subsection{A Hierarchy of Optimal Transport Denoisers}
In this paper, we study signal denoising under the Wasserstein metric and examine whether shrinkage aids denoising in a distributional sense. Recall that for $r \geq 1$, the $r$-Wasserstein metric \citep[Theorem 2.18]{villani2003topics} on the real line takes the closed-form expression,
\begin{align*}
W_r^r (P, Q) = \int_0^1 \left| F^{-1}(t) - G^{-1}(t) \right|^r \dd t
\end{align*}
where $F^{-1}, G^{-1}$ are the quantile functions of $P, Q$ respectively.

Considering the denoiser of the form $Y + h(Y)$, we look for denoising function $h: \R \rightarrow \R$ that achieves significantly smaller error than the trivial denoiser $Y$, measured by the Wasserstein metric, namely,
\begin{align*}
 W_r \big( \mathcal{L}(Y + h(Y)), \mathcal{L}(X) \big) \stackrel{?}{\ll} W_r \big( \mathcal{L}(Y), \mathcal{L}(X) \big)  \;.
\end{align*}
This is analogous to Stein's paradox, but now at the distributional level. Under mild regularity conditions, the optimal denoiser in the Wasserstein sense is the unique optimal transport map pushing $Q$ onto $P$ \citep{villani2003topics}, denoted by $T_\infty: \R \to \R$:
\begin{align*}
T_\infty(y) := F^{-1} \circ G(y)\;, \quad \text{with}\quad T_\infty \sharp Q = P \;.
\end{align*}

We will establish a complete hierarchy of denoisers, $T_0 \to T_1 \to \ldots \to T_\infty$, connecting from the trivial denoiser, $T_0(Y) = Y$, all the way to the optimal transport denoiser, $T_\infty(Y) \stackrel{\cL}{=} X$. The denoisers $T_K, K = 1, 2, \ldots$ are constructed using polynomials of higher-order score functions of $Q$, and achieve progressively better denoising accuracy.

To see how such a hierarchy is constructed, we note that for a monotonically increasing map $y \mapsto y + h(y)$,
\begin{align*}
W_r^r \big( \mathcal{L}(Y + h(Y)), \mathcal{L}(X) \big) &= \int_{0}^{1} \left| (id + h) \circ G^{-1}(t) - F^{-1}(t) \right|^r \dd t  \\
&= \int_{\R} \left| y + h(y) - F^{-1} \circ G(y) \right|^r q(y) \dd y \;.
\end{align*}
We see that the optimal transport denoiser $T_{\infty} := F^{-1} \circ G$ minimizes the above functional, and we establish that it admits a series expansion in terms of the noise level $\eta = \sigma^2/2$, namely,
\begin{align*}
T_{\infty}(y) = F^{-1} \circ G(y) = y + \sum_{k=1}^\infty \frac{\eta^k}{k!} h_k(y) \;.
\end{align*}
The functions $h_k : \R \to \R, k \geq 1$ are higher-order denoising functions that refine the map progressively at noise resolution $\sigma^{2k}$. We define the hierarchy of denoisers $T_K, K = 1, 2, \ldots$ by truncating this infinite expansion.

Let $q^{(k)}: \R \to \R$ denote the $k$-th order derivative of $q: \R \to \R$, the density function of $Q$. Each denoising function $h_k: \R \to \R$ in the expansion can be solved via a system of polynomial recursions involving higher-order score functions of $q$, defined as $\frac{q^{(m)}}{q}, 1\leq m \leq 2k-1$ \citep{bobkov2024fisher}. Specifically, the first few terms are given by:
\begin{align*}
h_1 &= \frac{G^{(2)}}{G^{(1)}} \\
h_2 &= - \frac{G^{(4)}}{G^{(1)}} + 2 \frac{G^{(3)}}{G^{(1)}} h_1 - \frac{G^{(2)}}{G^{(1)}} h_1^2\\
h_3 &= \frac{G^{(6)}}{G^{(1)}} - 3 \frac{G^{(5)}}{G^{(1)}} h_1 + 3 \frac{G^{(4)}}{G^{(1)}} h_1^2 + 3 \frac{ G^{(3)}}{G^{(1)}} h_2  - \frac{ G^{(3)}}{G^{(1)}} h_1^3 - 3 \frac{ G^{(2)}}{G^{(1)}} h_1 h_2 \\
h_4 &= -\frac{G^{(8)}}{G^{(1)}} + 4\frac{G^{(7)}}{G^{(1)}}h_1 - 6\frac{G^{(6)}}{G^{(1)}}h_1^2 - 6\frac{G^{(5)}}{G^{(1)}}h_2  + 4\frac{G^{(5)}}{G^{(1)}}h_1^3 + 12\frac{G^{(4)}}{G^{(1)}}h_1 h_2 + 4\frac{G^{(3)}}{G^{(1)}}h_3   \\
& \quad \quad - \frac{G^{(4)}}{G^{(1)}}h_1^4 - 6\frac{G^{(3)}}{G^{(1)}}h_1^2 h_2 - \frac{G^{(2)}}{G^{(1)}} (4h_1 h_3  + 3 h_2^2 )\;.
\end{align*}
The structure of these formulas is governed by a combinatorial relationship involving Bell polynomials; see Theorem~\ref{thm:ot-expansion-G} for the general formula for $h_k, k \geq 1$:
\begin{align*}
h_k = &- (-1)^{k} \frac{G^{(2k)}}{G^{(1)}} - \sum_{i=1}^{k-1} (-1)^{k-i} \binom{k}{i} \sum_{j=0}^{i-1} \frac{G^{(2k-i-j)}}{G^{(1)}} B_{i,i-j}(h_1, \ldots, h_{j+1}) \\
& \qquad \qquad - \sum_{j=0}^{k-2} \frac{G^{(k-j)}}{G^{(1)}} B_{k,k-j}(h_1, \ldots, h_{j+1}) \;.
\end{align*}
Here $B_{n,k}$'s are the partial Bell polynomials \citep{comtet1974advanced} as in Definition~\ref{def:bell-polynomials}, capturing the combinatorial structure of integer partitions. The above recursions can be solved sequentially to obtain $h_k, k \geq 1$.

We emphasize that such denoising functions $h_k, k\geq 1$ only depend on higher-order score functions of $Q$, and are \emph{agnostic} to the specific form or knowledge of the signal distribution $P$. This phenomenon motivates the name \emph{agnostic denoisers}. This is in the spirit of the James-Stein denoiser, which is also agnostic to the prior distribution $P$; however, the James-Stein denoiser is only agnostic for discrete distributions with $d$ atoms. In contrast, our denoisers $T_K, K \geq 1$ are agnostic across all distributions on $\R$.

\paragraph{Organization of Results}
The paper is organized as follows. As a warm-up for the technical combinatorial structure, Section~\ref{sec:F-expansion} and Theorem~\ref{thm:ot-expansion-F} study the infinite expansion of the optimal transport map $T_\infty = F^{-1} \circ G$ in terms of higher-order score functions of the signal distribution $P$. Section~\ref{sec:higher-order-asymptotics} and Theorem~\ref{thm:denoiser-accuracy} quantify $W_r(T_K \sharp Q, P)$, the denoising accuracy of $T_K, K \geq 1$, as well as the uniform error of the map $T_K$ in approximating the optimal transport map $T_{\infty} = F^{-1} \circ G$. Section~\ref{sec:noise-asymptotics} is reminiscent of g-modeling in empirical Bayes \citep{efron2014two}, where one uses prior information to construct denoisers. Though the results in Section~\ref{sec:noise-asymptotics} are primarily of theoretical interest, they prepare the technical tools for the main results in Section~\ref{sec:higher-order-denoising}.

Moving on to the main results, Section~\ref{sec:higher-order-denoising} derives and studies optimal transport denoisers via higher-order score functions of $Q$, directly identifiable and estimable based on observations $Y$. Section~\ref{sec:G-expansion} and Theorem~\ref{thm:ot-expansion-G} represent $\{h_k\}_{k=1}^\infty$ as polynomials of higher-order score functions $\{ \tfrac{q^{(m)}}{q} \}_{m=1}^\infty$. This result reveals how higher-order score functions encode information of the optimal transport denoising map, a finding new to the literature. We then study estimation strategies and rates of convergence for these higher-order scores, based on $n$ i.i.d. samples $\{Y_i\}_{i=1}^n$. We present two approaches with theoretical guarantees on the estimation accuracy: (i) a plug-in estimation approach via Gaussian kernel smoothing that estimates $q^{(m)}(y)$ locally at any given $y$, see Section~\ref{sec:gaussian-kernel-smoothing} and Theorem~\ref{thm:kde-derivative-rate} for the rate of convergence; (ii) a direct estimation approach via higher-order score matching that estimates the function $\tfrac{q^{(m)}}{q}$ globally, see Section~\ref{sec:higher-order-score-matching} and Theorem~\ref{thm:score-matching-higher-order} for the rate of convergence. This is reminiscent of f-modeling in empirical Bayes \citep{efron2014two}.

\subsection{Relations to the Literature}

The optimal transport (OT) denoisers are fundamentally different from the Bayes-optimal or empirical Bayes denoisers \citep{efron1973stein,robbins1964empirical,james1961estimation} that minimize the mean squared error. Indeed, the Bayes-optimal denoiser takes the form
\begin{align*}
T^{\rm bayes}(y) =  y + \sigma^2 \frac{q'(y)}{q(y)} \;,
\end{align*}
different from the hierarchy of OT denoisers $T_K, K \geq 0$: with $T_0(y) = y$, and
\begin{align*}
T_1(y) &= y + \eta h_1(y) = y + \frac{\sigma^2}{2}  \frac{q'(y)}{q(y)}  \;,\\
T_2(y) &= y + \eta h_1(y) + \frac{\eta^2}{2} h_2(y) = y + \frac{\sigma^2}{2}  \frac{q'(y)}{q(y)} - \frac{\sigma^4}{8} \left(  \frac{1}{2} \left(\frac{q'(y)}{q(y)} \right)^2 + \left( \frac{q'(y)}{q(y)}  \right)'  \right)' \;.
\end{align*}
Denoiser $T_1$ has been studied extensively in diffusion models; see \citep{liang2024denoising,beyler2025optimal}.
\citet{liang2025distributional} studies these two OT denoisers $T_1, T_2$ in a general $d$-dimensional setting, and establishes that these two denoisers are \emph{universal denoisers}. They remain universal to a wide range of noise distributions, extending beyond the Gaussian setting that requires Tweedie's formula. On the one hand, this paper generalizes the results in \citep{liang2025distributional} and derives a complete hierarchy of OT denoisers for arbitrary order $K$. On the other hand, this paper considers only $d=1$ and the Gaussian noise.

Empirical Bayes methods \citep{efron1973stein,robbins1964empirical,james1961estimation,ghosh2025stein} study the Bayes-optimal denoiser constructed based on i.i.d. samples $\{Y_i\}_{i=1}^n$ drawn from $Q$. There are two distinct approaches, called f-modeling (modeling on the observed data, the $Y$ space) and g-modeling (modeling on the unobserved prior/signal, or the $X$ space) \citep{efron2014two}; we refer the reader to \citep[Chapter 6]{ignatiadis2025empirical} for detailed discussions. \citet{efron2014two} states that g-modeling on the $X$ space has dominated the theoretical literature in empirical Bayes, see \citep{efron2016empirical, jiang2009general,laird1978nonparametric} for examples. In particular, \citet{jiang2009general} studied nonparametric maximum likelihood estimation (NPMLE) of the prior distribution $P$ based on i.i.d. samples $\{Y_i\}_{i=1}^n$ drawn from $Q$. They established the rate of convergence of NPMLE $\widehat{P}$ to the true prior $P$ in terms of the Hellinger distance between the induced densities on the $Y$ space. \citet{ghosh2025stein} draws the connection between score matching \citep{hyvarinen2005estimation} and SURE \citep{stein1981estimation,xie2012sure} to propose a SURE training procedure; they also establish the corresponding rate of convergence to the true prior measured by Fisher divergence.

We note that the theoretical empirical Bayes literature primarily focuses on estimating the prior $P$ first, then plugging in Tweedie's formula to construct the empirical Bayes denoiser. However, as noted in \citep{efron2014two}, f-modeling directly on the $Y$ space is more natural for constructing denoisers, since the data $\{Y_i\}_{i=1}^n$ are directly observed. \citet{efron2014two} proposed several parametric models on the $Y$ space to estimate the Bayes-optimal denoiser.

\emph{This paper takes a fundamentally different approach from the empirical Bayes literature}: we construct denoisers directly on the observational $Y$ space via higher-order score functions, without estimating the prior $P$ on the $X$ space. The denoisers $T: \R \rightarrow \R$ constructed in this paper are optimal in the Wasserstein sense, and when applied to $Y \sim Q$, the denoised distribution $T \sharp Q$ matches the signal distribution $P$ up to high accuracy. In other words, this paper proposes directly estimating the optimal transport map $T$ that pushes $Q$ onto $P$, without the knowledge or estimation of $P$. This is in sharp contrast to the g-modeling approach, which first estimates $P$, then uses that information to construct empirical Bayes denoisers. Our denoisers are nonparametric in nature, agnostic to the specific form of the prior distribution $P$. In terms of estimation of the higher-order score functions, we present two approaches: (i) a plug-in estimation via Gaussian kernel smoothing; (ii) a direct estimation approach via higher-order score matching. Both approaches come with theoretical guarantees on the rates of convergence. Crucially, our optimal transport approach discovers and isolates a notion of \emph{agnostic denoisers}: the denoisers $T_K, K \geq 1$ are \emph{agnostic to the specific form of $P$}; they only depend on the higher-order score functions of $Q$.

Score matching techniques have received considerable attention in the machine learning literature \citep{hyvarinen2005estimation,hyvarinen2007some, vincent2011connection,saremi2019neural}. The traditional score matching focuses on estimating the first-order score function $\tfrac{q^{(1)}}{q}$ \citep{hyvarinen2007some, saremi2018deep, sriperumbudur2017density, feng2024optimal} via Fisher divergence. \citet{bobkov2024fisher} recently studied Fisher-type information induced by higher-order score functions $\tfrac{q^{(m)}}{q}, m \geq 1$. We establish estimation strategies and rates of convergence for these higher-order scores, and derive through advanced combinatorics that these higher-order scores encode full information to construct the OT denoising map with $T \sharp Q = P$. 

To our knowledge, this paper is the first to fully characterize the optimal transport denoisers in terms of higher-order score functions; \emph{we make precise the connection among Fisher-type information (higher-order scores), optimal transport (Brenier map), and combinatorics (integer partitions) in this classic signal denoising problem.}

The closest works to ours are \citet{garcia2024new,jaffe2025constrained,liang2025distributional}, which study denoisers motivated by OT theory \citep{ambrosio2005gradient,villani2008optimal}. Motivated by the fact that Bayes-optimal and empirical Bayes denoisers typically shrink the distribution too aggressively, these three papers propose different remedies. \citet{garcia2024new} and \citet{jaffe2025constrained} assume the existence of certain structural information of the prior distribution $P$, and solve for denoisers that minimize the mean squared error $\E \| T(Y) - X \|^2$ subject to such distributional constraints, say $T \sharp Q = P$; one can also view them as post-processing empirical Bayes denoisers to avoid over-shrinkage for better distributional matching.

In contrast, this paper and \citet{liang2025distributional} do not require any a priori knowledge of $P$, and construct optimal transport-inspired denoisers solely based on the noisy measurements $Y \sim Q$. This paper generalizes the results in \citep{liang2025distributional} to arbitrary higher-order denoisers $T_K, K \geq 1$ in the univariate setting. This paper characterizes the fundamental role of higher-order score functions in constructing the hierarchy of OT denoisers $T_1, T_2, \ldots, T_\infty$, and reveals the combinatorial structure underlying the infinite expansion of the optimal transport map via Bell polynomial recursions.

\subsection{Notation}
We use the following notations throughout the paper. We use lower-case letters such as $x, y, z$ to denote scalars in $\R$, and upper-case letters such as $X, Y, Z$ to denote random variables in $\R$. We reserve $f, g, h : \R \to \R$ and $F, G, T: \R \to \R$ to denote measurable functions or maps. $P, Q$ denote probability distributions on $\R$, with density functions $p, q: \R \rightarrow \R_{\geq 0}$ respectively with respect to the Lebesgue measure; $\phi: \R \rightarrow \R_{\geq 0}$ denotes the standard normal density function. $\sigma > 0$ denotes the noise level, and we call $\eta = \sigma^2/2$ the noise parameter.

For a smooth function $f: \R \to \R$ and an integer $k$, we use $f^{(k)}$ to denote the $k$-th order derivative of $f$. For a probability distribution $P$ on $\R$ with density $p$, we use $\mathcal{L}(X)$ to denote the distribution of the random variable $X \sim P$. We use $T \sharp Q$ to denote the push-forward distribution of $Q$ via a measurable map $T: \R \to \R$, i.e., for any measurable set $A \subseteq \R$, $(T \sharp Q)(A) = P_{Y \sim Q}(T(Y) \in A)$. For two probability distributions $P, Q$ on $\R$, we use $W_r(P, Q)$ to denote the $r$-Wasserstein distance between $P$ and $Q$ with $r\geq 1$. Given a function $f$ and a probability distribution $Q$, we denote the $L^2(Q)$-norm as $\| f\|_{L^2(Q)} = \left( \E_{Y \sim Q} f(Y)^2 \right)^{1/2}$. We use $\| f\|_{\infty} = \sup_{x} |f(x)|$ to denote the $L^\infty$-norm of $f$, $\| f \|_{\mathrm{Lip}}$ to denote the uniform Lipschitz constant for $f$.

Given i.i.d. samples $\{Y_i\}_{i=1}^n$, we denote the empirical expectation $\widehat{\E}_n f = \tfrac{1}{n} \sum_{i=1}^n f(Y_i)$. $\cF, \cH$ denote function classes. We use $\cN(\epsilon, \mathcal{F}, \|\cdot\|)$ to denote the covering number of a function class $\mathcal{F}$ at scale $\epsilon$ under the norm $\|\cdot\|$. We use the standard asymptotic notations $f(n) = O(g(n))$, or equivalently $f(n) \precsim g(n)$, if there exists a constant $C > 0$ such that $|f(n)| \leq C |g(n)|$ for all sufficiently large $n$; we write $f(n) \asymp g(n)$ if $f(n) \precsim g(n)$ and $g(n) \precsim f(n)$.

\section{Optimal Transport Expansion and Noise Asymptotics}
\label{sec:noise-asymptotics}

This section derives an infinite expansion of the optimal transport map pushing the distribution of noisy measurements $Q$ to the distribution of signal $P$, in terms of the noise parameter $\eta = \sigma^2/2$. We discover that such an expansion satisfies a combinatorial recursive structure that involves polynomials of the higher-order score functions. We name this infinite expansion the noise asymptotics of the optimal transport map.

We note that it is without loss of generality to assume $\sigma \in [0, 1]$ and thus $\eta = \sigma^2/2 \in [0, 1/2]$: denote $\gamma^2 := \mathrm{Var}[Y]$, then we can rescale $Y$ to $\widetilde{Y} = Y/\gamma = \widetilde{X} + \widetilde{\sigma} Z$ where $\widetilde{X} = X/\gamma$ and $\widetilde{\sigma} = \sigma/\gamma \in [0, 1]$. It follows that the noise level $\sigma$ can always be normalized to be at most $1$ by rescaling the data.

We start with the model setup.
\begin{definition}[Model]
	\label{def:model-setup}
	Let $X \in \R$ be a random variable with distribution $P$ and a valid density $p$ with respect to the Lebesgue measure. Let $Z \sim \mathcal{N}(0,1)$ be independent standard normal noise. Consider the noisy measurement	$Y = X + \sigma Z$ with distribution $Q$ and density $q$. Let $\eta = \sigma^2/2$.

	Let $F, G : \R \to [0, 1]$ be the cumulative distribution functions of $P, Q$ respectively, defined as
	\begin{align*}
	F(x) = \int_{-\infty}^x p(z) \dd z \;, \quad G(y) = \int_{-\infty}^y q(z) \dd z \;.
	\end{align*}
	The notations $F^{(k)}$ and $G^{(k)}$ denote the $k$-th order derivatives of $F, G$ respectively. $F^{-1}, G^{-1}: [0, 1] \to \R$ denote the quantile functions (generalized inverse) of $F$ and $G$ respectively, defined as
	\begin{align*}
	F^{-1}(t) = \inf \{ x \in \R: F(x) > t \} \;, \quad G^{-1}(t) = \inf \{ y \in \R: G(y) > t \} \;.
	\end{align*}
\end{definition}

\subsection{F-expansion of the Optimal Transport Map}
\label{sec:F-expansion}
This section derives the noise asymptotic expansion of the optimal transport map pushing $Q$ to $P$ in terms of the derivatives of the cumulative distribution function $F$ of $P$.

\begin{definition}[Optimal Transport Map]
	\label{def:ot-map}
Let $P, Q$ be two probability distributions on $\R$ with densities $p, q$ respectively. Let $F^{-1}$ be the quantile function of $P$, and $G$ be the cumulative distribution function of $Q$.

The optimal transport map $T: \R \to \R$ pushing $Q$ to $P$ is defined as
\begin{align*}
T(y) = F^{-1} \circ G(y) \;.
\end{align*}
Such a map exists and is unique almost everywhere when $P, Q$ are absolutely continuous with respect to the Lebesgue measure on $\R$ \citep[Theorem 2.18]{villani2003topics}, and minimizes the $r$-Wasserstein distance between $P$ and $Q$ for any $r \geq 1$,
\begin{align*}
W_{r} (P, Q) = \inf_{T: T\sharp Q = P}~ \left(\int_{\R} | T(y) - y |^{r} q(y) \dd y \right)^{1/r} \;.
\end{align*}
\end{definition}

The noise asymptotic expansion requires a recursion defined by partial Bell polynomials \citep{comtet1974advanced}, which we first introduce below.
\begin{definition}[Bell Polynomials]
	\label{def:bell-polynomials}
The partial Bell polynomials $B_{n,k}(x_1, x_2, \ldots, x_{n-k+1})$ are defined as
\begin{align*}
B_{n,k}(x_1, x_2, \ldots, x_{n-k+1}) = \sum \frac{n!}{j_1! j_2! \cdots j_{n-k+1}!} \prod_{\ell=1}^{n-k+1}\left( \frac{x_\ell}{\ell!} \right)^{j_\ell}  \;,
\end{align*}
where the sum is taken over all sequences of non-negative integers $j_1, j_2, \ldots, j_{n-k+1}$ such that
\begin{align*}
j_1 + j_2 + \cdots + j_{n-k+1} = k \;, \\
j_1 + 2 j_2 + \cdots + (n-k+1) j_{n-k+1} = n \;.
\end{align*}
\end{definition}

With these definitions, we are ready to present the noise asymptotic expansion of the optimal map $T$ transporting $Q$ to $P$ in terms of the derivatives of $F$.

\begin{theorem}[F-expansion of the Optimal Transport Map]
	\label{thm:ot-expansion-F}
	Consider the additive Gaussian noise model in Definition~\ref{def:model-setup}. Assume that $F$ is infinitely differentiable. Define the following series expansion
	\begin{align}
		\label{eqn:T_F-series}
	T_F(y) := y + \sum_{k=1}^\infty \frac{\eta^k}{k!} g_k(y)
	\end{align}
	where $g_1, g_2, \ldots : \R \rightarrow \R$ are functions defined iteratively through the Bell polynomials $B_{n,k}$ in Definition~\ref{def:bell-polynomials}:
	\begin{align}
		\label{eqn:g_k-recursion}
		g_1(y) &= \frac{F^{(2)}(y)}{F^{(1)}(y)} \;, \quad k = 1 \nonumber \\
	g_k(y) &=  -  \sum_{j=2}^k  \frac{F^{(j)}(y)}{F^{(1)}(y)} \cdot B_{k,j}\big( g_1,\dots,g_{k-j+1} \big)(y) + \frac{F^{(2k)}(y)}{F^{(1)}(y)} \;, \quad k \geq 2 \;.
	\end{align}

	Then, for $y \in \R$ such that $T_F(y)$ is absolutely convergent, $T_F$ represents the optimal transport map pushing $Q$ onto $P$ as in Definition~\ref{def:ot-map}, namely,
	\begin{align*}
	 F^{-1} \circ G(y) = T_F(y) \;.
	\end{align*}
\end{theorem}
\begin{remark}

	This proof is deferred to Section~\ref{sec:proof-F-expansion}. This theorem characterizes how the optimal denoising map depends on the higher-order score functions $\tfrac{F^{(m)}}{F^{(1)}}, m \geq 2$ of the signal distribution $P$.
	Here, we evaluate a few initial terms explicitly using the recursion derived in this theorem.
\begin{enumerate}
\item For $k=1$, $$B_{1,1}(g_1) = g_1.$$ We have
$$g_1 = \frac{F^{(2)}}{F^{(1)}}.$$
\item For $k=2$, $$B_{2,1}(g_1, g_2) = g_2, \quad B_{2,2}(g_1) = g_1^2.$$
Equation~\eqref{eqn:g_k-recursion} reads $g_2 = - \frac{F^{(2)}}{F^{(1)}} B_{2,2}(g_1) + \frac{F^{(4)}}{F^{(1)}}$, and thus we have
$$g_2 = - \frac{F^{(2)}}{F^{(1)}} g_1^2 + \frac{F^{(4)}}{F^{(1)}}.$$
\item For $k=3$, $$B_{3,1}(g_1, g_2, g_3) = g_3, \quad B_{3,2}(g_1, g_2) = 3 g_1 g_2, \quad B_{3,3}(g_1) = g_1^3.$$
Equation~\eqref{eqn:g_k-recursion} reads $g_3 = - \frac{F^{(2)}}{F^{(1)}} B_{3,2}(g_1, g_2) -  \frac{F^{(3)}}{F^{(1)}} B_{3,3}(g_1) + \frac{F^{(6)}}{F^{(1)}}$, and thus we have
$$g_3 = -  \frac{F^{(2)}}{F^{(1)}} (3g_1 g_2) - \frac{F^{(3)}}{F^{(1)}} g_1^3 + \frac{F^{(6)}}{F^{(1)}}.$$
\item For $k=4$,
$$B_{4,1} = g_4, \quad B_{4,2} = 4 g_1 g_3 + 3 g_2^2, \quad B_{4,3} = 6 g_1^2 g_2, \quad B_{4,4} = g_1^4.$$
Equation~\eqref{eqn:g_k-recursion} reads
$g_4 = - \tfrac{F^{(2)}}{F^{(1)}} B_{4,2} -  \tfrac{F^{(3)}}{F^{(1)}} B_{4,3} -  \tfrac{F^{(4)}}{F^{(1)}} B_{4,4} + \tfrac{F^{(8)}}{F^{(1)}}$, and thus we have
\begin{align*}
g_4 = & - \frac{F^{(2)}}{F^{(1)}} (4 g_1 g_3 + 3 g_2^2) -  \frac{F^{(3)}}{F^{(1)}} (6 g_1^2 g_2) -  \frac{F^{(4)}}{F^{(1)}} g_1^4 + \frac{F^{(8)}}{F^{(1)}}.
\end{align*}
\end{enumerate}
\end{remark}

\subsection{Higher-Order Noise Asymptotics}
\label{sec:higher-order-asymptotics}
In this section, we derive the accuracy of higher-order denoisers based on the noise asymptotic expansion of the optimal transport map in Theorem~\ref{thm:ot-expansion-F}. Define the $K$-th order denoiser  $T_K: \mathrm{supp}(P) \to \R$ as the truncated series expansion of $T$ up to the $K$-th order term,
\begin{align}
T_K(y) = y + \sum_{k=1}^K \frac{\eta^k}{k!} g_k(y) \;.
\end{align}
Since $\{g_i \}_{i=1}^K$ depend on higher-order score functions $\tfrac{F^{(k)}}{F^{(1)}}, k \leq 2K$, which are well-defined on the support of $P$, we concern ourselves only with the behavior of $T_K$ on the support of $P$.

For a monotonically increasing map $T: \mathrm{supp}(P) \to \R$, we define the Wasserstein distance between $T \sharp Q$ and $P$ restricted to the support of $P$ as
\begin{align}
	\label{eqn:Wasserstein-restricted}
W_r^\ast ( T \sharp Q, P) := \left( \int_{\mathrm{supp}(P)} \left| T(y) - F^{-1} \circ G(y) \right|^r q(y) \dd y \right)^{1/r} \;.
\end{align}
In the case when $\mathrm{supp}(P) = \R$, this definition coincides with the standard $r$-Wasserstein distance $W_r ( T \sharp Q, P)$; see Lemma~\ref{lem:wasserstein-real-line} for this equivalence.

The next theorem studies the $W_r^\ast(\cdot, \cdot)$ between the denoised distribution $T_K \sharp Q$ and the signal distribution $P$, as well as the uniform approximation error between the denoiser $T_K$ and the optimal transport map $F^{-1} \circ G$ restricted to the support of $P$.
\begin{theorem}[Higher-Order Accuracy]
	\label{thm:denoiser-accuracy}
	Let $K\geq 1$ and consider the additive Gaussian noise model in Definition~\ref{def:model-setup} with $F$ that is $(2K+2)$-differentiable. Let $T_K$ be the $K$-th order denoiser defined above with $\{g_i\}_{i=1}^K$ defined in Theorem~\ref{thm:ot-expansion-F}. Assume that $\mathrm{supp}(P)$ is compact and $\inf_{x \in \mathrm{supp}(P)} F^{(1)}(x)>0$.  Recall \eqref{eqn:Wasserstein-restricted}, for any $r \geq 1$, we have
\begin{align*}
W_r^\ast \big( T_K \sharp Q, P) \precsim \eta^{K+1} \;,
\end{align*}
and the uniform approximation error
\begin{align*}
	\sup_{y \in \mathrm{supp}(P)} \left| T_K(y) - F^{-1} \circ G(y) \right| \precsim \eta^{K+1} \;.
\end{align*}
Here the constants in the $\precsim$ notation only depend on the uniform upper bound of $\sup_{x \in \mathrm{supp}(P)} | F^{(k)} (x)|$ for $k \leq 2K+2$ and the uniform lower bound of $\inf_{x \in \mathrm{supp}(P)} F^{(1)}(x)$.
\end{theorem}

\begin{remark}

	As noted earlier, we can assume that $\sigma \in [0, 1]$ without loss of generality since linear rescaling of $Y$ does not fundamentally change the problem. Therefore, $\eta = \sigma^2/2 \in [0, 1/2]$. In this case, the rate in Theorem~\ref{thm:denoiser-accuracy} contracts faster than $(1/2)^{K+1}$ as $K$ increases.
\end{remark}

\begin{remark}

We note that $T_K$ is monotonically increasing when $\eta$ is small. Indeed, since $T_K$ is a finite sum of smooth functions, it suffices to show that the first derivative $\dv{y} T_K(y)$ is positive when $\eta$ is small enough. Note that
\begin{align*}
\dv{y} T_K(y) = 1 + \sum_{k=1}^K \frac{\eta^k}{k!} \dv{y} g_k(y) \;.
\end{align*}
Therefore, when $\eta$ is sufficiently small, $\dv{y} T_K(y) > 0$ for all $y \in \mathrm{supp}(P)$ compactly supported.
\end{remark}

\begin{proof}[Proof of Theorem~\ref{thm:denoiser-accuracy}]
	By the definition \eqref{eqn:Wasserstein-restricted}, we have
\begin{align*}
W_r^\ast \big( T_K \sharp Q, P) &= \left( \int_{\mathrm{supp}(P)} \left| T_K(y) - F^{-1} \circ G(y) \right|^r q(y) \dd y \right)^{1/r} \;.
\end{align*}
From Lemma~\ref{lem:finite-term-accuracy},
\begin{align*}
	\left| F \circ T_K(y) -  G(y) \right| & \leq C \cdot \eta^{K+1} \;,
\end{align*}
uniformly for $y \in \mathrm{supp}(P)$. Here the constant $C$ only depends on the upper bound of $\sup_{y \in \mathrm{supp}(P)}| F^{(k)} (y)|$ for $k \leq 2K+2$ and the lower bound of $\inf_{y \in \mathrm{supp}(P)} F^{(1)}(y)$.

Under the assumption that $\inf_{y \in \mathrm{supp}(P)} F^{(1)}(y) > 0$, we have that $F^{-1}$ is Lipschitz continuous on $[0,1]$. Therefore, uniformly over $y \in \mathrm{supp}(P)$,
\begin{align*}
\left| T_K(y) - F^{-1} \circ G(y) \right| \leq \| F^{-1} \|_{\mathrm{Lip}} \cdot \left| F \circ T_K(y) -  G(y) \right| \leq \tilde{C} \cdot \eta^{K+1} \;.
\end{align*}
Here the constant $\tilde{C}$ is a universal constant that does not depend on $\eta$ or $y$; it only depends on $\| F^{-1} \|_{\mathrm{Lip}}$ and $\| F^{(k)} \|_{\infty}$ for $k \leq 2K+2$.
\end{proof}

\section{Higher-Order Score Denoising: Identification and Estimation}
\label{sec:higher-order-denoising}
The previous section establishes the noise-asymptotic expansion of the optimal transport map pushing $Q$ to $P$. However, the expansion in Theorem~\ref{thm:ot-expansion-F} depends on the higher-order score functions of $P$, unknown in practice. This section presents an alternative approach that quantifies the denoising expansion in terms of the higher-order score functions of $Q$, which are identifiable and estimable from the noisy measurements $Y$.

Denoisers derived in this section are agnostic to the specific form of the prior/signal distribution $P$. They only depend on the higher-order score functions of the noisy measurement distribution $Q$. We also present two estimation strategies for these higher-order score functions of $Q$, along with their convergence rates.

To formulate and study the estimation results in this section, we first introduce the class of H\"older smooth functions.
\begin{definition}[H\"older Smooth Functions]
	\label{def:holder-smooth}
For $\alpha > 0$, let $\lfloor \alpha \rfloor$ be the largest integer strictly less than $\alpha$. For $L > 0$, we define the class of H\"older smooth functions with smoothness parameter $\alpha$ and radius $L$ as
\begin{align*}
\cH_{\Omega}^{\alpha}(L) := \left\{ f : \Omega \to \R ~\mid~ \max_{0 \leq k \leq \lfloor \alpha \rfloor}  \sup_{x \in \Omega} | f^{(k)}(x)| + \sup_{x \neq y \in \Omega} \frac{|f^{(\lfloor \alpha \rfloor)}(x) - f^{(\lfloor \alpha \rfloor)}(y)|}{|x - y|^{\alpha - \lfloor \alpha \rfloor}} \leq L \right\} \;.
\end{align*}
We write $\cH^{\alpha}(L)$ when the support $\Omega = \R$ the whole real line.
\end{definition}

\subsection{G-expansion and Denoising Equations}
\label{sec:G-expansion}
This section provides an alternative expansion of the optimal transport map pushing $Q$ to $P$ in terms of polynomials of higher-order score functions $\tfrac{G^{(k)}}{G^{(1)}}$.
\begin{theorem}[G-expansion of the Optimal Transport Map]
	\label{thm:ot-expansion-G}
Consider the same setting as Theorem~\ref{thm:ot-expansion-F}, we define the following series expansion
\begin{align}
	\label{eqn:T_G-series}
T_G(y) := y + \sum_{k=1}^\infty \frac{\eta^k}{k!} h_k(y)
\end{align}
where $h_1, h_2, \ldots$ are functions defined iteratively through the Bell polynomials $B_{n,k}(h_1, h_2, \ldots, h_{n-k+1})$ in Definition~\ref{def:bell-polynomials}:
\begin{align}
	\label{eqn:h_k-recursion}
\sum_{l=0}^{k-1} (-1)^l \frac{k!}{(k-l)! l!} \sum_{j=1}^{k-l} G^{(2l+j)}(y) \cdot B_{k-l,j}\big( h_1,\dots,h_{k-l-j+1} \big)(y) + (-1)^k G^{(2k)}(y) = 0.
\end{align}
In particular, the above equation uniquely determines $h_k$ in terms of $h_1, \ldots, h_{k-1}$, with the only term involving $h_k$ in \eqref{eqn:h_k-recursion} being $G^{(1)}(y) h_k(y)$ ($l=0, j=1$); and, each $h_k$ is a polynomial function of $\tfrac{G^{(m)}}{G^{(1)}}, m \leq 2k$.

Then, for $y \in \R$ such that $T_G(y)$ is absolutely convergent, $T_G$ is the optimal transport map pushing $Q$ onto $P$ as in Definition~\ref{def:ot-map}, namely,
\begin{align*}
 F^{-1} \circ G(y) = T_G(y) \;.
\end{align*}
\end{theorem}

\begin{remark}

	The crucial difference between Theorem~\ref{thm:ot-expansion-G} and Theorem~\ref{thm:ot-expansion-F} is that the expansion in Theorem~\ref{thm:ot-expansion-G} depends on the higher-order score functions of $Q$, which are identifiable and estimable from the noisy measurements $Y$. This makes Theorem~\ref{thm:ot-expansion-G} practical for denoising applications.
We evaluate a few initial terms explicitly before presenting the proof.
\begin{enumerate}
\item For $k=1$, $$B_{1,1}(h_1) = h_1.$$

Equation~\eqref{eqn:h_k-recursion} reads $G^{(1)} B_{1, 1}(h_1) - G^{(2)}  = 0$, and thus we have
$$h_1 = \frac{G^{(2)}}{G^{(1)}}.$$

\item For $k=2$, $$B_{2,1}(h_1, h_2) = h_2, \quad B_{2,2}(h_1) = h_1^2.$$

Equation~\eqref{eqn:h_k-recursion} reads $G^{(1)} B_{2, 1}(h_1, h_2) + G^{(2)} B_{2, 2}(h_1) - 2 G^{(3)} B_{1, 1}(h_1) + G^{(4)} = 0$, and thus we have
$$h_2 = - \frac{G^{(4)}}{G^{(1)}} + 2 \frac{G^{(3)}}{G^{(1)}} h_1 - \frac{G^{(2)}}{G^{(1)}} h_1^2.$$

\item For $k=3$, $$B_{3,1}(h_1, h_2, h_3) = h_3, \quad B_{3,2}(h_1, h_2) = 3 h_1 h_2, \quad B_{3,3}(h_1) = h_1^3.$$

Equation~\eqref{eqn:h_k-recursion} reads $G^{(1)} B_{3, 1}(h_1, h_2, h_3) + G^{(2)} B_{3, 2}(h_1, h_2) + G^{(3)} B_{3, 3}(h_1) - 3 G^{(3)} B_{2, 1}(h_1, h_2) - 3 G^{(4)} B_{2, 2}(h_1) + 3 G^{(5)} B_{1, 1}(h_1) - G^{(6)} = 0$, and thus we have
$$h_3 = \frac{G^{(6)}}{G^{(1)}} - 3 \frac{G^{(5)}}{G^{(1)}} h_1 + 3 \frac{G^{(4)}}{G^{(1)}} h_1^2 + 3 \frac{ G^{(3)}}{G^{(1)}} h_2  - \frac{ G^{(3)}}{G^{(1)}} h_1^3 - 3 \frac{ G^{(2)}}{G^{(1)}} h_1 h_2.$$

\item For $k=4$,
$$B_{4,1} = h_4, \quad B_{4,2} = 4 h_1 h_3 + 3 h_2^2, \quad B_{4,3} = 6 h_1^2 h_2, \quad B_{4,4} = h_1^4.$$

Substituting into Equation~\eqref{eqn:h_k-recursion} and simplifying, we have
\begin{align*}
h_4 = &-\frac{G^{(8)}}{G^{(1)}} + 4\frac{G^{(7)}}{G^{(1)}}h_1 - 6\frac{G^{(6)}}{G^{(1)}}h_1^2 - 6\frac{G^{(5)}}{G^{(1)}}h_2  + 4\frac{G^{(5)}}{G^{(1)}}h_1^3 + 12\frac{G^{(4)}}{G^{(1)}}h_1 h_2 + 4\frac{G^{(3)}}{G^{(1)}}h_3   \\
&  \quad \quad - \frac{G^{(4)}}{G^{(1)}}h_1^4 - 6\frac{G^{(3)}}{G^{(1)}}h_1^2 h_2 - \frac{G^{(2)}}{G^{(1)}} (4h_1 h_3  + 3 h_2^2 )\;.
\end{align*}

\end{enumerate}
\end{remark}

\begin{proof}[Proof of Theorem~\ref{thm:ot-expansion-G}]
We plug in $T = T_G$ and quantify the difference in terms of an integral form of the static Monge-Amp\`ere equation \citep{caffarelli1992regularity,liang2025distributional}, namely $F \circ T_G(y) - G(y)$. If this difference is zero for $y$ such that $T_G(y)$ is absolutely convergent, then by the uniqueness of the optimal transport map \citep[Theorem 2.18]{villani2003topics}, we conclude that the optimal transport map $F^{-1}\circ G (y) = T_G(y)$.

Recall if the series expansion of $T_G(y) = y + \sum_{k=1}^\infty \frac{\eta^k}{k!} h_k(y)$ is absolutely convergent, we can apply Lemma~\ref{lem:F-in-G-series} that expands $F$ in terms of derivatives of $G$, and then by Fubini's theorem, we have
\begin{align*}
F\circ T_G(y) &= \sum_{l=0}^\infty (-1)^l \frac{\eta^l}{l!} G^{(2l)} \circ T_G(y) \\
& = \sum_{l=0}^\infty (-1)^l \frac{\eta^l}{l!} \sum_{j=0}^\infty  G^{(2l+j)}(y) \frac{1}{j!}  \left( \sum_{k=1}^\infty \frac{\eta^k}{k!} h_k(y) \right)^j \;.
\end{align*}

Apply Lemma~\ref{lem:bell-poly} that relates powers of series to Bell polynomials, we have again by Fubini's theorem,
\begin{align*}
F\circ T_G(y) &= \sum_{l=0}^\infty (-1)^l \frac{\eta^l}{l!} \sum_{j=0}^\infty G^{(2l+j)}(y) \sum_{n \geq j} \frac{\eta^n}{n!} B_{n,j}\big( h_1,\dots,h_{n-j+1} \big)(y)  \\
&= \sum_{k=0}^\infty \frac{\eta^k}{k!} \sum_{l=0}^k (-1)^l \frac{k!}{(k-l)! l!} \sum_{j=0}^{k-l} G^{(2l+j)}(y) \cdot B_{k-l,j}\big( h_1,\dots,h_{k-l-j+1} \big)(y)  \\
& = G(y) + \sum_{k=1}^\infty \frac{\eta^k}{k!} \sum_{l=0}^k (-1)^l \frac{k!}{(k-l)! l!} \sum_{j=0}^{k-l} G^{(2l+j)}(y) \cdot B_{k-l,j}\big( h_1,\dots,h_{k-l-j+1} \big)(y) \;.
\end{align*}
Here, the last line uses the fact that for the term $k = 0$, we have $B_{0,0} = 1$.

By the definition of $h_k, k\geq 1$ in \eqref{eqn:h_k-recursion}, we have
\begin{align*}
	\sum_{l=0}^{k-1} (-1)^l \frac{k!}{(k-l)! l!} \sum_{j=1}^{k-l} G^{(2l+j)}(y) \cdot B_{k-l,j}\big( h_1,\dots,h_{k-l-j+1} \big)(y) +  (-1)^k G^{(2k)}(y)  = 0 \;, \\
 \text{and thus} \quad \sum_{l=0}^{k} (-1)^l \frac{k!}{(k-l)! l!} \sum_{j=0}^{k-l} G^{(2l+j)}(y) \cdot B_{k-l,j}\big( h_1,\dots,h_{k-l-j+1} \big)(y) = 0 \;.
\end{align*}
Here the last equality uses the fact that $B_{k,0} = \mathbb{I}(k=0)$.

Putting things together, we have shown that for $y$ such that $T_G(y)$ is absolutely convergent, $F\circ T_G(y) - G(y) = 0.$
\end{proof}

Following the arguments in Theorem~\ref{thm:ot-expansion-G} and Lemma~\ref{lem:finite-term-accuracy}, we have the following corollary on the finite-order denoising accuracy, measured in terms of the integral form of the static Monge-Amp\`ere equation \citep{caffarelli1992regularity,villani2008optimal,deb2025no,liang2025distributional}, namely $|F \circ T_K (y) - G(y)|$ at any given $y \in \R$. The proof is deferred to Section~\ref{sec:proof-G-expansion}. As noted earlier, we can assume that $\sigma \in [0, 1]$ without loss of generality since linear rescaling of $Y$ does not change the denoising problem. Therefore, we can always assume $\eta = \sigma^2/2 \in [0, 1/2]$.

\begin{corollary}
	\label{cor:finite-term-accuracy-G}
	Consider the same setting as Theorem~\ref{thm:ot-expansion-G} and $\{h_k\}_{k\geq 1}$ defined therein.
For any integer $K \geq 1$, define the $K$-th order denoiser as
\begin{align*}
T_K(y) := y + \sum_{k=1}^K \frac{\eta^k}{k!} h_k(y) \;.
\end{align*}
Assume that $F, G \in \cH^{2K+2}(L)$ for some $L > 0$. Then, for any $y \in \R$,
\begin{align*}
	\left| F \circ T_K(y) -  G(y) \right|  \leq C \cdot \eta^{K+1} \;.
\end{align*}
Here the constant $C$ only depends on $L$, $K$, and $q(y)$.
\end{corollary}

Theorem~\ref{thm:ot-expansion-G} and Corollary~\ref{cor:finite-term-accuracy-G} provide an iterative refinement scheme to construct higher-order denoisers based on polynomials of $\tfrac{G^{(k+1)}}{G^{(1)}}, k\geq 1$. Since $\tfrac{G^{(k+1)}}{G^{(1)}} = \tfrac{q^{(k)}}{q}$ is the higher-order score function of $Q$,  it is identifiable from the noisy observations $\{Y_i\}_{i=1}^n$ drawn i.i.d. from $Q$.

In the hierarchy of denoisers $\{T_K\}_{K\geq 1}$ defined above, we start with the trivial denoiser $T_0(y) = y$, then use the first-order score function of $G$ to refine $T_1(y) = y + \eta \frac{G^{(2)}(y)}{G^{(1)}(y)}$. Theorem~\ref{thm:ot-expansion-G} generalizes this procedure to arbitrary order $K$, with each refinement step involving a polynomial in the higher-order score functions $\frac{G^{(m)}(y)}{G^{(1)}(y)}$ up to order $m \leq 2K$. The recursions are defined using Bell-type polynomials in \eqref{eqn:h_k-recursion}. The denoising accuracy improves at rate $\eta^{K}$, as quantified in Corollary~\ref{cor:finite-term-accuracy-G} (and Theorem~\ref{thm:denoiser-accuracy}). In the limit $K\rightarrow \infty$, this denoiser recovers the optimal transport map $T_K \rightarrow F^{-1} \circ G$, without requiring the knowledge of $F$, the signal distribution.

The following sections study estimation strategies and rates for these higher-order scores, based on $n$ i.i.d. samples $\{Y_i\}_{i=1}^n$. We present two approaches with theoretical guarantees on the estimation accuracy: (i) plug-in estimation via Gaussian kernel smoothing; and (ii) direct estimation via higher-order score matching.

\subsection{Plug-in Estimation via Gaussian Kernel Smoothing}
\label{sec:gaussian-kernel-smoothing}

Recall that the $K$-th order denoiser is given by
\begin{align*}
T_K(y) = y + \sum_{k=1}^K \frac{\eta^k}{k!} h_k(y) \;,
\end{align*}
where $h_1, h_2, \ldots, h_K$ are defined by the Bell-polynomial-type recursion in Theorem~\ref{thm:ot-expansion-G} that involves only the derivatives of $\{G^{(m)}\}_{m=1}^{2K}$. Therefore, to estimate $T_K$, it suffices to estimate the density function $q$ and its derivatives up to order $2K-1$.

Given i.i.d. samples from $Q$, denoted as $\{ Y_i \}_{i=1}^n$, we can estimate $q$ and its derivatives using Gaussian kernel smoothing.
Given a bandwidth $b>0$, we define the following estimator for $q^{(m)}$, the $m$-th derivative of $q$:
\begin{align}
	\label{eqn:kde-derivative}
\widehat{q}^{(m)}_b(y) = \frac{1}{n b^{m+1}} \sum_{i=1}^n \phi^{(m)}\left( \frac{y - Y_i}{b} \right) \;,
\end{align}
where $\phi: \R \rightarrow \R$ is a standard Gaussian kernel function.

\begin{theorem}[Estimation Rates for $q^{(m)}$: Gaussian Kernel Smoothing]
	\label{thm:kde-derivative-rate}
	Let $m \geq 0$ be an integer.
Let $\phi(z) = \frac{1}{\sqrt{2\pi}} e^{-\frac{z^2}{2}}$ be the standard Gaussian kernel. Suppose the density $q \in \cH^{m+2}(L)$ for some $L > 0$. Then, with the bandwidth choice $b_m \asymp n^{-\frac{1}{2m+5}}$, we have the following mean squared error bound for any fixed $y \in \R$,
\begin{align*}
\E \left|\widehat{q}^{(m)}_{b_m}(y) - q^{(m)}(y)\right|^2 \precsim n^{-\frac{4}{2m+5}} \;,
\end{align*}
where the constant in the $\precsim$ notation only depends on $L$ and $m$.
\end{theorem}
\begin{proof}[Proof of Theorem~\ref{thm:kde-derivative-rate}]
The proof is a direct application of the bias-variance decomposition
\begin{align*}
	\E \left|\widehat{q}^{(m)}_{b}(y) - q^{(m)}(y)\right|^2 &= \left( \E[\widehat{q}^{(m)}_{b}(y)] - q^{(m)}(y) \right)^2 + \mathrm{Var}\left[\widehat{q}^{(m)}_{b}(y)\right] \\
	& \leq \frac{L^2}{4} \cdot b^4 + \frac{L m!}{\sqrt{2\pi}} \cdot \frac{1}{n b^{2m+1}} \;,
\end{align*}
where the last line uses Lemma~\ref{lem:kde-bias-variance}.
With the choice $b = b_m = \left( \sqrt{\tfrac{8}{\pi}}\tfrac{m!}{L}\right)^{\frac{1}{2m+5}} \cdot n^{-\frac{1}{2m+5}}$, we finish the proof.
\end{proof}

\begin{remark}

One can rationalize the estimation rate as follows: the optimal rate for estimating the $m$-th derivative of a density function of smoothness $\alpha = m+2$ is at the rate
\begin{align*}
\E \left|\widehat{q}^{(m)}_{b_m}(y) - q^{(m)}(y)\right|^2 \precsim  n^{-\frac{2(m+2 - m)}{2(m+2)+1}} \;.
\end{align*}

With this analytic estimation method, we can form the ratio estimate by the plug-in approach, with $\widehat{q}^{(m)}_{b_m}(y)$ and $\widehat{q}^{(0)}_{b_0}(y)$ being the estimators for $q^{(m)}(y)$ and $q(y)$ respectively, where $b_m \asymp n^{-\frac{1}{2m+5}}$ and $b_0 \asymp n^{-\frac{1}{5}}$ are the optimal bandwidths for estimating $q^{(m)}(y)$ and $q(y)$ respectively.
\begin{align*}
\widehat{\frac{q^{(m)}(y)}{q(y)}} := \frac{\widehat{q}^{(m)}_{b_m}(y)}{\widehat{q}^{(0)}_{b_0}(y)}  \;.
\end{align*}
\end{remark}

\subsection{Direct Estimation via Higher-Order Score Matching}
\label{sec:higher-order-score-matching}

Another practical approach is to estimate the higher-order score function $f^\ast_m(y) := \frac{q^{(m)}(y)}{q(y)}$ directly through generalizations of score matching \citep{hyvarinen2005estimation,hyvarinen2007some}. Unlike the local approach of estimating $q^{(m)}(y)$ and $q(y)$ separately for a given $y$ and then forming the ratio, score matching provides a global estimation approach that learns the entire function $f^\ast_m$.

Let $\mathcal{F}$ be a function class that contains a higher-order score function $f^\ast_m(y)$. The generalized score matching estimator for the $m$-th order score function is defined as
\begin{align}
	\label{eqn:score-matching-estimator}
\widehat{f}_m \in \argmin_{f \in \mathcal{F}} \widehat{\E}_n \left[ \frac{1}{2} f(Y)^2 + (-1)^{m+1} f^{(m)}(Y) \right] \;,
\end{align}
where $\widehat{\E}_n$ is the empirical expectation over the i.i.d. samples $\{ Y_i \}_{i=1}^n$ drawn from $Q$. In the case $m=1$, this reduces to the classical score matching estimator in \citep{hyvarinen2005estimation}.

\begin{theorem}[Estimation Rates for $q^{(m)}/q$: Higher-Order Score Matching]
	\label{thm:score-matching-higher-order}
	Let $m \geq 1$ be an integer.
	Let $\Omega \subset \R$ be a compact interval, and let $\cF = \cH^{\alpha}_{\Omega}(L)$ be the class of H\"older smooth functions with smoothness parameter $\alpha > m$. Assume that the $m$-th order score function $\frac{q^{(m)}}{q} \in \cF$ and that $q$ has vanishing boundaries up to order $m-1$.
	Let $\widehat{f}_m$ be the $m$-th order score matching estimator defined in \eqref{eqn:score-matching-estimator}, then we have the following mean squared error bound,
	\begin{align*}
	\E \left\| \widehat{f}_m - \frac{q^{(m)}}{q} \right\|_{L_2(Q)}^2 \precsim  \begin{cases}
 n^{-(\alpha - m)}, & \alpha  <  m + \frac{1}{2} \\
 n^{-\frac{1}{2}} \log n, & \alpha  = m +\frac{1}{2} \\
 n^{-\frac{1}{2}}, & \alpha > m + \frac{1}{2}
	\end{cases} \;.
	\end{align*}
\end{theorem}
\begin{remark}

This theorem states that if the $m$-th order score function $\frac{q^{(m)}}{q}$ is at least $(m+\frac{1}{2})$-H\"older smooth, then the score matching estimator can estimate $\frac{q^{(m)}}{q}$ at a rate of $n^{-\frac{1}{2}}$, independent of the smoothness level $m$.

\end{remark}

\begin{proof}[Proof of Theorem~\ref{thm:score-matching-higher-order}]
We first apply the representation lemma for higher-order score functions (see Lemma~\ref{lem:higher-order-score-representation}) with $f = \widehat{f}_m$,
\begin{align}
	\label{eqn:population-risk-decomposition}
\E_{Y \sim Q} \frac{1}{2} \left| \widehat{f}_m(Y) - f^\ast_m(Y) \right|^2 &	= \E_{Y \sim Q} \left[ \tfrac{1}{2} \widehat{f}_m(Y)^2 + (-1)^{m+1} \widehat{f}_m^{(m)}(Y) \right] \nonumber\\
& \quad - \E_{Y \sim Q} \left[ \tfrac{1}{2} f^\ast_m(Y)^2 + (-1)^{m+1} (f^\ast_m)^{(m)}(Y) \right] \;.
\end{align}
By the definition of the empirical risk minimizer $\widehat{f}_m$, and the fact that $f^\ast_m \in \mathcal{F}$, we have
\begin{align}
	\label{eqn:empirical-risk-comparison}
\widehat{\E}_n \left[ \frac{1}{2} \widehat{f}_m(Y)^2 + (-1)^{m+1} \widehat{f}_m^{(m)}(Y) \right] &\leq \widehat{\E}_n \left[ \frac{1}{2} f^{\ast}_m(Y)^2 + (-1)^{m+1} (f^\ast_m)^{(m)}(Y) \right]
\end{align}
With the above two displays \eqref{eqn:population-risk-decomposition} and \eqref{eqn:empirical-risk-comparison}, and the fact $ \widehat{f}_m \in \mathcal{F}$, we have
\begin{align*}
&\E_{Y \sim Q} \frac{1}{2} \left| \widehat{f}_m(Y) - f^\ast_m(Y) \right|^2 \\
&\leq \sup_{f \in \mathcal{F}} \left( \E - \widehat{\E}_n \right) \left[ \frac{1}{2} f^2 + (-1)^{m+1} f^{(m)} \right] - \left( \E - \widehat{\E}_n \right) \left[ \frac{1}{2} (f^{\ast}_m)^2 + (-1)^{m+1} (f^\ast_m)^{(m)} \right] \;.
\end{align*}

Due to Gin\'e-Zinn symmetrization (see Lemma~\ref{lem:symmetrization}), we have
\begin{align*}
	\E \| \widehat{f}_m - f^\ast_m \|_{L_2(Q)}^2 &= \E_{\{Y_i \}_{i=1}^n} \E_{Y \sim Q} \frac{1}{2} \left| \widehat{f}_m(Y) - f^\ast_m(Y) \right|^2 \\
	&\leq 4 \E_{\{Y_i \}_{i=1}^n} \E_{\{ \epsilon_i \}_{i=1}^n} \sup_{f \in \mathcal{F}} \frac{1}{n} \sum_{i=1}^n \epsilon_i \left[ \frac{1}{2} f^2(Y_i) + (-1)^{m+1} f^{(m)}(Y_i) \right] \;,
\end{align*}
where the last step uses the fact $f^\ast_m \in \mathcal{F}$.

By the Lipschitz contraction inequality of Rademacher complexities (see \citet{ledoux1991rademacher} or \citep[Lemma 2]{farrell2021deep}), we have
\begin{align*}
&\E_{\{Y_i \}_{i=1}^n} \E_{\{ \epsilon_i \}_{i=1}^n} \sup_{f \in \mathcal{F}} \frac{1}{n} \sum_{i=1}^n \epsilon_i \left[ \frac{1}{2} f^2(Y_i) + (-1)^{m+1} f^{(m)}(Y_i) \right] \\
&\leq L \cdot \E_{\{Y_i \}_{i=1}^n} \E_{\{ \epsilon_i \}_{i=1}^n} \sup_{f \in \mathcal{F}} \frac{1}{n} \sum_{i=1}^n \epsilon_i f(Y_i) + \E_{\{Y_i \}_{i=1}^n} \E_{\{ \epsilon_i \}_{i=1}^n} \sup_{f \in \mathcal{F}} \frac{1}{n} \sum_{i=1}^n \epsilon_i f^{(m)}(Y_i)
\end{align*}
The constant $L$ is due to the fact that $|\frac{1}{2} f^2 - \frac{1}{2} g^2| \leq L |f - g|$ for any $f, g \in \mathcal{F}$.

What remains is to bound the Rademacher complexities of the H\"older function class $\mathcal{F}$ and its $m$-th order derivative function class $\mathcal{F}^{(m)} := \{ f^{(m)}: f \in \mathcal{F} \}$.
Using standard entropy integral bounds and metric entropy of H\"older class and also \citep[Theorem 2.7.1]{van1996weak}, we know the metric entropy for $\cF = \cH^{\alpha}_{\Omega}(L)$ is upper bounded
$$\log \mathcal{N}(\epsilon, \cF, \| \cdot \|_\infty) \leq C_{L, \Omega, \alpha} \cdot \epsilon^{-\frac{1}{\alpha}} \;,$$
where the constant $C_{L, \alpha, \Omega} > 0$ only depends on $L, \alpha$ and the support $\Omega$, and does not depend on $\epsilon$. Therefore, for any $\{Y_i\}_{i=1}^n$, we can upper bound the Rademacher complexity of $\mathcal{F}$ using standard entropic integral bounds (see \citep[Theorem 4]{schreuder2020bounding} and \citep[Lemma 26]{liang2021well}):
\begin{align*}
&\E_{\{ \epsilon_i \}_{i=1}^n} \sup_{f \in \mathcal{F}} \frac{1}{n} \sum_{i=1}^n \epsilon_i f(Y_i) \\
&\leq \inf_{0 \leq \delta \leq L} \left( 4 \delta + \frac{12}{\sqrt{n}} \int_\delta^{L} \sqrt{ \log \mathcal{N}(\epsilon, \mathcal{F}, \| \cdot \|_\infty) } \dd \epsilon \right) \leq \begin{cases}
C_1 n^{-\alpha}, & \alpha < \frac{1}{2} \\
C_2 n^{-\frac{1}{2}} \log n, & \alpha = \frac{1}{2} \\
C_3 n^{-\frac{1}{2}}, & \alpha > \frac{1}{2}
\end{cases}
\end{align*}
for some universal constants $C_1, C_2, C_3 > 0$ that depend on $L, \alpha$ and $\Omega$. Here when $\alpha < \frac{1}{2}$, we choose $\delta \asymp n^{-\alpha}$ to optimize the bound; when $\alpha = \frac{1}{2}$, we choose $\delta \asymp n^{-\frac{1}{2}} \log n$; when $\alpha > \frac{1}{2}$, we can set $\delta \asymp n^{-\frac{1}{2}}$ or simply set $\delta = 0$ as the integral converges.

Similarly, we have $f^{(m)} \in \cH_{\Omega}^{\alpha - m}(L)$ is $(\alpha - m)$-H\"older smooth for any $f \in \mathcal{F}$, and thus
\begin{align*}
\E_{\{ \epsilon_i \}_{i=1}^n} \sup_{f \in \mathcal{F}} \frac{1}{n} \sum_{i=1}^n \epsilon_i f^{(m)}(Y_i)  \leq \begin{cases}
C_1 n^{-(\alpha - m)}, & \alpha < m + \frac{1}{2} \\
C_2 n^{-\frac{1}{2}} \log n, & \alpha =  m+ \frac{1}{2} \\
C_3 n^{-\frac{1}{2}}, & \alpha > m + \frac{1}{2}
\end{cases} \;.
\end{align*}
We now complete the proof.
\end{proof}

\section*{Acknowledgements}
The author acknowledges the support from the NSF Career Award (DMS-2042473) and the Wallman Society of Fellows at the University of Chicago. The author also thanks Nabarun Deb, Nikos Ignatiadis, Ben Recht, and Richard Samworth for comments on an earlier draft. This work is part II of a two-part series on distributional shrinkage; part I can be found in \cite{liang2025distributional}.

\ifbiblatex
  \printbibliography
\fi

\appendix

\section{Technical Proofs}

\subsection{Proofs for Section~\ref{sec:F-expansion}}
\label{sec:proof-F-expansion}

We first present a few technical lemmas.

\begin{lemma}
	\label{lem:G-in-F-series}
	Consider the model in Definition~\ref{def:model-setup} with an analytic $F$, we have
$G(y) = \sum_{k=0}^\infty \frac{\eta^k }{k!} F^{(2k)}(y) \;.$
\end{lemma}
\begin{proof}[Proof of Lemma~\ref{lem:G-in-F-series}]
Recall that $G(y) = \int F(y - \sigma z) \phi(z) \dd z$, where $\phi(z)$ is the standard normal density function. Therefore,
\begin{align*}
G(y) = \sum_{l=0}^\infty \frac{\sigma^{2l} \E[Z^{2l}]}{(2l)!} F^{(2l)}(y) = \sum_{l=0}^\infty \frac{\eta^l}{l!} F^{(2l)}(y) \;.
\end{align*}
\end{proof}

\begin{lemma}
	\label{lem:F-in-G-series}
	Consider the model in Definition~\ref{def:model-setup} with an analytic $F$, we have
$F(y) = \sum_{k=0}^\infty \frac{(-\eta)^k }{k!} G^{(2k)}(y) \;.$
\end{lemma}
\begin{proof}[Proof of Lemma~\ref{lem:F-in-G-series}]
The proof proceeds by recalling Lemma~\ref{lem:G-in-F-series}, namely $G^{(2k)}(y) = \sum_{l=0}^\infty \frac{\eta^l}{l!} F^{(2k+2l)}(y)$. We have
\begin{align*}
\sum_{k=0}^\infty \frac{(-\eta)^k }{k!} G^{(2k)}(y) &= \sum_{k=0}^\infty \frac{(-\eta)^k }{k!} \sum_{l=0}^\infty \frac{\eta^l}{l!} F^{(2(k+l))}(y) \\
&= \sum_{m=0}^\infty \left( \sum_{k=0}^m \frac{(-1)^k m!}{k! (m-k)!} \right) \frac{\eta^m}{m!} F^{(2m)}(y) \\
& = \sum_{m=0}^\infty (1 + (-1))^m \frac{\eta^m}{m!} F^{(2m)}(y) = F(y) \;.
\end{align*}
\end{proof}

\begin{lemma}
	\label{lem:bell-poly}
	Recall the definition of Bell polynomials $B_{n,j}(x_1, x_2, \ldots, x_{n-j+1})$ in Definition~\ref{def:bell-polynomials}, we have
	\begin{align*}
	\frac{1}{j!} \left( \sum_{k=1}^\infty  \frac{\eta^k}{k!} x_k  \right)^j = \sum_{n \geq j} \frac{\eta^n}{n!} B_{n,j}(x_1, x_2, \ldots, x_{n-j+1}) \;.
	\end{align*}
\end{lemma}

\begin{proof}[Proof of Lemma~\ref{lem:bell-poly}]
\begin{align*}
\frac{1}{j!} \left( \sum_{k=1}^\infty  \frac{\eta^k}{k!} x_k  \right)^j &= \frac{1}{j!} \sum_{i_1+i_2+\cdots = j} \frac{j!}{i_1! i_2! \cdots } \prod_{m=1}^{\infty} \left( \frac{\eta^m}{m!} x_m \right)^{i_m} \\
&= \sum_{n \geq j} \frac{\eta^n}{n!} \sum_{\substack{i_1+i_2+\cdots = j \\ 1 \cdot i_1 + 2 \cdot i_2 + \cdots  = n}} \frac{n!}{i_1! i_2! \cdots i_{n-j+1}!} \prod_{m=1}^{n-j+1} \left( \frac{x_m}{m!} \right)^{i_m} \\
&= \sum_{n \geq j} \frac{\eta^n}{n!} B_{n,j}(x_1, x_2, \ldots, x_{n-j+1}) \;.
\end{align*}
\end{proof}

Now we are ready to prove Theorem~\ref{thm:ot-expansion-F}.
\begin{proof}[Proof of Theorem~\ref{thm:ot-expansion-F}]
The proof follows by plugging the infinite expansion of $T = T_F$ in the identity $F \circ T (y) = G(y)$ and matching the coefficients of $\eta^k$ for each $k \geq 0$. Then, by the uniqueness of the optimal transport map, we conclude the desired expansion.
We start with the series expansion of the analytic function $F$: under the assumption that the series $T_F(y)$ is absolutely convergent for a fixed $y$, we have by Fubini's theorem,
\begin{align*}
F\circ T_F(y) &= \sum_{j=0}^\infty  F^{(j)}(y) \frac{1}{j!} \left( \sum_{k=1}^\infty \frac{\eta^k}{k!} g_k(y) \right)^j  \\
& = \sum_{j=0}^\infty  F^{(j)}(y) \sum_{n \geq j} \frac{\eta^n}{n!} B_{n,j}\big( g_1,\dots,g_{n-j+1} \big)(y)  \\
&= \sum_{k=0}^\infty \frac{\eta^k}{k!} \sum_{j=0}^k F^{(j)}(y) \cdot B_{k,j}\big( g_1,\dots,g_{k-j+1} \big)(y) \;.
\end{align*}
With the definition of $g_1, g_2, \ldots$ in \eqref{eqn:g_k-recursion}, and the fact $B_{k, 0} = \mathbb{I}(k=0)$, we have
\begin{align*}
\sum_{j=0}^k F^{(j)}(y) \cdot B_{k,j}\big( g_1,\dots,g_{k-j+1} \big)(y) =
\begin{cases}
 F^{(0)}(y) \;, & k = 0 \\
 F^{(1)}(y) B_{1, 1}(g_1) = F^{(2)}(y) \;, & k  = 1 \\
 \sum_{j=1}^k F^{(j)} B_{k,j}\big( g_1,\dots,g_{k-j+1} \big) =  F^{(2k)} \;, & k \geq 2
\end{cases} \;,
\end{align*}
where the last line uses the fact $ B_{k,1}\big( g_1,\dots,g_k \big) =g_k$.

As a result, we have the following identity
\begin{align*}
F\circ T_F(y)  = \sum_{k=0}^\infty \frac{\eta^k}{k!} \sum_{j=0}^k F^{(j)}(y) \cdot B_{k,j}\big( g_1,\dots,g_{k-j+1} \big)(y) = \sum_{k=0}^\infty \frac{\eta^k}{k!} F^{(2k)}(y) = G(y) \;.
\end{align*}
Here, the last line uses Lemma~\ref{lem:G-in-F-series}. By the uniqueness of the optimal transport map, we conclude that $F^{-1} \circ G(y) = T_F(y)$ for all $y$ such that the series expansion is absolutely convergent.
\end{proof}

\subsection{Proofs for Section~\ref{sec:higher-order-asymptotics}}

\begin{lemma}
\label{lem:wasserstein-real-line}
Let $P, Q$ be two probability distributions with valid densities $p, q$, and $T_0:\R \to \R$ be a monotonically increasing map. Then for any $r \geq 1$,
\begin{align*}
W_r^r ( T_0 \sharp Q, P) = \int_{\R} | T_0(y) - F^{-1} \circ G(y) |^r q(y) \dd y \;.
\end{align*}
\end{lemma}
\begin{proof}[Proof of Lemma~\ref{lem:wasserstein-real-line}]
	For a measurable map $T_0: \R \to \R$, denote the push-forward distribution $T_0 \sharp Q$ as the distribution of the random variable $T_0(Y)$ where $Y \sim Q$. Let $F_{T_0 \sharp Q}^{-1}$ be the quantile function of the distribution $T_0 \sharp Q$. Then $F_{T_0 \sharp Q}^{-1}(t)  = T_0 \circ G^{-1}(t)$ for $t \in [0, 1]$ for monotonically increasing $T_0$.
By definition of the Wasserstein distance in $1$-dimension, we have
\begin{align*}
W_r^r ( T_0 \sharp Q, P) &= \int_{0}^{1} | F_{T_0 \sharp Q}^{-1}(t) - F^{-1}_{P}(t) |^r \dd t = \int_{0}^{1} | T_0 \circ G^{-1}(t) - F^{-1}(t) |^r \dd t \\
&= \int_{\R} | T_0(y) - F^{-1} \circ G(y) |^r q(y) \dd y \;.
\end{align*}
Here, the last line uses the change of variable $t = G(y)$.
\end{proof}

\begin{lemma}
	\label{lem:finite-term-accuracy}
	 Under the same setting as in Theorem~\ref{thm:denoiser-accuracy},
\begin{align*}
	\sup_{y \in \mathrm{supp}(P)} \left| F \circ T_K(y) -  G(y) \right| &\leq C \cdot \eta^{K+1} \;.
\end{align*}
Here the universal constant $C >0$ depends on the upper bounds of $F^{(k)}(y)$ for $k \leq 2K+2$ and the lower bound of $p(y) = F^{(1)}(y)$ for $y \in \mathrm{supp}(P)$.
\end{lemma}
\begin{proof}[Proof of Lemma~\ref{lem:finite-term-accuracy}]
First, recall Lemma~\ref{lem:G-in-F-series},
$$
\left| G(y) - \sum_{k=0}^K \frac{\eta^{k}}{k!} F^{(2k)}(y) \right| \leq C_1 \cdot \eta^{K+1} \;,$$
for some constant $C_1 > 0$ that depends on the upper bounds of $\sup_{y \in \mathrm{supp}(P)} |F^{(2K+2)}(y)|$.

Next, we expand $F \circ T_K(y)$ using Taylor expansion,
\begin{align}
	\label{eqn:taylor-expansion-F-TK}
\left| F\circ T_K(y) - \sum_{j=0}^K F^{(j)}(y) \frac{1}{j!} \left( \sum_{k=1}^K \frac{\eta^k}{k!} g_k(y) \right)^j \right| \leq   C_2  \cdot \eta^{K+1}
\end{align}
Here the constant $C_2 > 0$ depends on the upper bounds of $\sup_{y \in \mathrm{supp}(P)} |F^{(K+1)}(y)|$ and $\sup_{y \in \mathrm{supp}(P)} |g_k(y)|$ for $k = 1, 2, \ldots, K$. Recall the definition of $g_k(y)$ in Theorem~\ref{thm:denoiser-accuracy}, we know that these are finite-degree polynomials of $\{ F^{(k)}/F^{(1)} \}_{k=2}^{2K}$. Under the assumption that $\sup_{y \in \mathrm{supp}(P)} |F^{(k)}(y)|$ and $\inf_{y \in \mathrm{supp}(P)} |F^{(1)}(y)|$ bounded, we have $g_k$'s uniformly bounded.

Note the second term in \eqref{eqn:taylor-expansion-F-TK} is a finite polynomial of $\eta$ up to degree $K^2$. We collect the coefficients for the monomials $\eta^k$ for $k = 0, 1, \ldots, K$,
\begin{align*}
&\eta^k \sum_{j=0}^{k} F^{(j)}(y) \frac{1}{j!} \sum_{\substack{i_1+i_2+\cdots = j \\ 1 \cdot i_1 + 2 \cdot i_2 + \cdots  = k}} \frac{j!}{i_1!i_2!\cdots i_{k-j+1}}  \left( \frac{g_1}{1!} \right)^{i_1} \left( \frac{g_2}{2!} \right)^{i_2} \cdots \left( \frac{g_{k-j+1}}{(k-j+1)!} \right)^{i_{k-j+1}}  \\
&= \frac{\eta^k}{k!} \sum_{j=0}^{k} F^{(j)}(y) \cdot B_{k,j}\big( g_1,\dots,g_{k-j+1} \big)(y)  = \frac{\eta^{k}}{k!} F^{(2k)}(y)\;.
\end{align*}
Therefore, all monomials of $\eta$ up to degree $K$ match those of $G(y)$, and the lemma follows.
\begin{align*}
\left| F \circ T_K(y) -  G(y) \right| &\leq \underbrace{\left| F\circ T_K(y) - \sum_{j=0}^K \frac{F^{(j)}(y)}{j!} \left( \sum_{k=1}^K \frac{\eta^k}{k!} g_k(y) \right)^{\!j} \right|}_{C_1 \cdot \eta^{K+1}} \\
& \quad + \underbrace{\left| G(y) - \sum_{k=0}^K \frac{\eta^{k}}{k!} F^{(2k)}(y) \right|}_{C_2 \cdot \eta^{K+1}} \\
& \quad + \underbrace{\left| \sum_{j=0}^K \frac{F^{(j)}(y)}{j!} \left( \sum_{k=1}^K \frac{\eta^k}{k!} g_k(y) \right)^{\!j} - \sum_{k=0}^K \frac{\eta^{k}}{k!} F^{(2k)}(y) \right|}_{C_3 \cdot \eta^{K+1}} \\
& \leq (C_1 + C_2 + C_3)\cdot \eta^{K+1} \;.
\end{align*}
Here $C_3 > 0$ is a universal constant as the coefficients of $\eta^k$ for $k = 0, 1, \ldots, K$ match for these two finite-degree polynomials of $\eta$. By the same logic, we know $C_1, C_2, C_3 > 0$ depend on the upper bounds of $\{ F^{(k)}(y) \}_{k=1}^{2K+2}$ and the lower bound of $F^{(1)}(y)$ for $y \in \mathrm{supp}(P)$.
\end{proof}

\subsection{Proofs for Section~\ref{sec:G-expansion}}
\label{sec:proof-G-expansion}
\begin{lemma}
	\label{lem:F-in-G-finite-term}
	Consider the model in Definition~\ref{def:model-setup} and assume $F(y)$ has derivatives up to order $2K+2$ uniformly bounded over $y \in \R$. Then
\begin{align*}
	F(y) - \sum_{l=0}^K (-1)^l \frac{\eta^l}{l!} G^{(2l)}(y) = O(\eta^{K+1}) \;.
\end{align*}
\end{lemma}
\begin{proof}[Proof of Lemma~\ref{lem:F-in-G-finite-term}]
Note from Lemma~\ref{lem:G-in-F-series}, $$G^{(2l)}(y) = \sum_{j=0}^{K-l} \frac{\eta^l}{l!} F^{(2l+2j)}(y) + O(\eta^{K+1}).$$
\begin{align*}
\sum_{l=0}^K (-1)^l \frac{\eta^l}{l!} G^{(2l)}(y) &= \sum_{l=0}^K (-1)^l \frac{\eta^l}{l!} \sum_{j=0}^{K-l} \frac{ \eta^j}{j!} F^{(2l+2j)}(y) + O(\eta^{K+1})  \\
&= \sum_{k=0}^K \left( \sum_{l=0}^k (-1)^{l} \frac{k!}{l! (k-l)!} \right) \frac{\eta^k}{k!} F^{(2k)}(y) + O(\eta^{K+1}) \\
&= F(y) + O(\eta^{K+1}) \;.
\end{align*}
\end{proof}

\begin{proof}[Proof of Corollary~\ref{cor:finite-term-accuracy-G}]
First, we note that by Lemma~\ref{lem:F-in-G-finite-term}, if $F^{(k)}$ are bounded for $k \leq 2K+2$, then
\begin{align*}
	F(T_K(y)) - \sum_{l=0}^K (-1)^l \frac{\eta^l}{l!} G^{(2l)}(T_K(y)) = O(\eta^{K+1}) \;.
\end{align*}
And note that when $G^{(k)}$ are uniformly bounded for $k \leq 2K+2$, we have
\begin{align*}
G^{(2l)}(T_K(y)) &= \sum_{j=0}^{K-l}  G^{(2l+j)}(y) \frac{1}{j!}  \left( \sum_{k=1}^K \frac{\eta^k}{k!} h_k(y) \right)^j + O(\eta^{K+1}) \;.
\end{align*}

Following the logic behind the calculations in Theorem~\ref{thm:ot-expansion-G}, we have
\begin{align*}
\sum_{l=0}^K (-1)^l \frac{\eta^l}{l!} G^{(2l)}(T_K(y)) &= \sum_{l=0}^K (-1)^l \frac{\eta^l}{l!} \sum_{j=0}^{K-l}  G^{(2l+j)}(y) \frac{1}{j!}  \left( \sum_{k=1}^K \frac{\eta^k}{k!} h_k(y) \right)^j + O(\eta^{K+1}) \\
&= \sum_{k=0}^K  \frac{\eta^k}{k!} \sum_{l=0}^{k} (-1)^l \binom{k}{l} \sum_{j=0}^{k-l} G^{(2l+j)}(y) B_{k-l,j}\big( h_1,\dots,h_{k-l-j+1} \big)(y) \\
&\quad + O(\eta^{K+1}) \\
&= G(y) + O(\eta^{K+1}) \;.
\end{align*}
The last line uses the defining equation~\eqref{eqn:h_k-recursion} of $h_k$.
\end{proof}

\subsection{Proofs for Section~\ref{sec:gaussian-kernel-smoothing}}
\label{sec:proof-gaussian-kernel-smoothing}

\begin{lemma}[Bias and Variance of Gaussian Smoothing]
	\label{lem:kde-bias-variance}
Let $\phi(z) = \frac{1}{\sqrt{2\pi}} e^{-\frac{z^2}{2}}$ be the standard Gaussian kernel. Recall the definition of the kernel density estimator in~\eqref{eqn:kde-derivative}. Suppose $q \in \cH^{m+2}(L)$ for some $L > 0$, then we have
\begin{align*}
\left| \E[\widehat{q}^{(m)}_b(y)] - q^{(m)}(y) \right|^2 &\leq  \frac{L^2}{4} \cdot b^4 \;. \\
\mathrm{Var}\left[\widehat{q}^{(m)}_b(y)\right] &\leq \frac{L m!}{\sqrt{2\pi}} \cdot \frac{1}{n b^{2m+1}} \;.
\end{align*}
\end{lemma}

\begin{proof}[Proof of Lemma~\ref{lem:kde-bias-variance}]
We start with the bias calculation
\begin{align*}
\left| \E[\widehat{q}^{(m)}_b(y)] - q^{(m)}(y) \right| &=  \left| \int \frac{1}{b^{m+1}} \phi^{(m)}\left( \frac{y - x}{b} \right) q(x) \dd x - q^{(m)}(y) \right| \\
&= \left| \int \frac{1}{b^{m}} \phi^{(m)}(z) q(y - b z) \dd z - q^{(m)}(y) \right| \\
& = \left| \int  \frac{(-1)^m}{b^{m}} \phi(z) \dv[m]{}{z} q(y - b z) \dd z - q^{(m)}(y) \right| \\
& = \left| \int \phi(z) q^{(m)}(y - b z)  \dd z -  \int \phi(z) q^{(m)}(y) \dd z \right| \\
& =  \left| \int \phi(z) \left( q^{(m)}(y - b z)  - q^{(m)}(y) + q^{(m+1)}(y) b z \right) \dd z \right|  \\
& \leq \int \phi(z) \left| q^{(m)}(y - b z)  - q^{(m)}(y) + q^{(m+1)}(y) b z \right| \dd z  \\
& \leq \int \phi(z) \frac{L b^2 z^2}{2} \dd z = \frac{L b^2}{2} \;,
\end{align*}
where the second last step uses the fact $\int z \phi(z) \dd z = 0$ and then by Taylor expansion of $q^{(m)}(y - b z)$ around $y$, and the last step uses the H\"older smoothness of $q$.

Now we calculate the variance,
\begin{align*}
\mathrm{Var}[\widehat{q}^{(m)}_b(y)] &= \frac{1}{n} \mathrm{Var}\left[ \frac{1}{b^{m+1}} \phi^{(m)}\left( \frac{y - Y_1}{b} \right) \right]
\end{align*}
and we bound
\begin{align*}
 \mathrm{Var}\left[ \frac{1}{b^{m+1}} \phi^{(m)}\left( \frac{y - Y_1}{b} \right) \right] &\leq \int \frac{1}{b^{2m+2}} \left( \phi^{(m)}\left( \frac{y - x}{b} \right) \right)^2 q(x) \dd x \\
&= \int \frac{1}{b^{2m+1}} \left( \phi^{(m)}(z) \right)^2 q(y - b z) \dd z \;.
\end{align*}

By the property of Gaussian derivatives, we know
\begin{align*}
	\phi^{(m)}(z) = (-1)^m He_m(z) \phi(z)
\end{align*}
where $He_m(z)$ is the $m$-th Hermite polynomial. Therefore,
\begin{align*}
&\int \frac{1}{b^{2m+1}} \left( \phi^{(m)}(z) \right)^2 q(y - b z) \dd z \\
&= \int \frac{1}{b^{2m+1}} He_m^2(z) \phi^2(z) q(y - b z) \dd z \\
&\leq \sup_{u, z} \left\{ q(u - bz) \phi(z) \right\} \cdot  \int \frac{1}{b^{2m+1}} He_m^2(z) \phi(z)  \dd z \\
&\leq \frac{L}{\sqrt{2\pi}} \frac{1}{b^{2m+1}} \int He_m^2(z) \phi(z) \dd z  = \frac{L m!}{\sqrt{2\pi}} \frac{1}{b^{2m+1}} \;,
\end{align*}
where the last step uses the fact $\int He_m^2(z) \phi(z) \dd z = m!$.
\end{proof}

\subsection{Proofs for Section~\ref{sec:higher-order-score-matching}}
\label{sec:proof-higher-order-score-matching}

\begin{lemma}[Higher-order Score Representation]
	\label{lem:higher-order-score-representation}
	Let $q$ be the density function of $Q$ with vanishing boundary conditions up to order $m-1$. For any $f$ that has continuous derivatives up to order $m$, we have
\begin{align*}
\E_{Y \sim Q} \frac{1}{2} \left| f(Y) - f^\ast_m(Y) \right|^2 &= \E_{Y \sim Q} \left[ \tfrac{1}{2} f(Y)^2 + (-1)^{m+1} f^{(m)}(Y) \right] \\
&\quad - \E_{Y \sim Q} \left[ \tfrac{1}{2} f^\ast_m(Y)^2 + (-1)^{m+1} (f^\ast_m)^{(m)}(Y) \right] \;.
\end{align*}
\end{lemma}
\begin{proof}[Proof of Lemma~\ref{lem:higher-order-score-representation}]
	For any given function $f$ that satisfies the conditions in the lemma, we have
	\begin{align*}
	\E_{Y \sim Q} \frac{1}{2} \left| f(Y) - f^\ast_m(Y) \right|^2 & = \int \frac{1}{2} \left( f(y) - \frac{q^{(m)}(y)}{q(y)} \right)^2 q(y) \dd y \\
	&= \int \left[ \frac{1}{2} f(y)^2 q(y) - f(y) q^{(m)}(y) + \frac{1}{2} \frac{q^{(m)}(y)^2}{q(y)} \right] \dd y \\
	&=\E_{Y \sim Q}\left[ \frac{1}{2} f(Y)^2 + (-1)^{m+1} f^{(m)}(Y) \right]  + \int \frac{1}{2} \frac{q^{(m)}(y)^2}{q(y)}  \dd y \;.
	\end{align*}
	Here, the last step applies integration by parts $m$ times and uses the vanishing boundary conditions up to order $m-1$.
	Plugging in $f = f^\ast_m$, we have the higher-order Fisher-type information \citep{bobkov2024fisher} identity
	\begin{align*}
	 \int \frac{1}{2} \frac{q^{(m)}(y)^2}{q(y)}  \dd y  &= - \E_{Y \sim Q} \left[ \frac{1}{2} f^\ast_m(Y)^2 + (-1)^{m+1} (f^\ast_m)^{(m)}(Y) \right] \;.
	\end{align*}
	Combining these two equations yields the desired result.
\end{proof}

\begin{lemma}[Symmetrization]
	\label{lem:symmetrization}
	\begin{align*}
	\E_{\{Y_i \}_{i=1}^n} \sup_{g \in \mathcal{G}} (\E - \widehat{\E}_n)[g(Y)] \leq 2 \E_{\{Y_i \}_{i=1}^n} \E_{\{ \epsilon_i \}_{i=1}^n} \sup_{g \in \mathcal{G}} \frac{1}{n} \sum_{i=1}^n \epsilon_i g(Y_i)
	\end{align*}
	where $\{ \epsilon_i \}_{i=1}^n$ are i.i.d. Rademacher random variables independent of $\{ Y_i \}_{i=1}^n$.
\end{lemma}
\begin{proof}[Proof of Lemma~\ref{lem:symmetrization}]
	This is a standard result in empirical process theory; see \citep{gine1990bootstrapping}, and \citep[Lemma 26]{liang2021well}. We include the argument for completeness. The proof follows from the observation that $\E[g(Y)] = \E_{\{Y_i' \}_{i=1}^n} \widehat{\E}_n'[g(Y)]$, where $\{Y_i' \}_{i=1}^n$ is an independent copy of $\{Y_i \}_{i=1}^n$, and thus
	\begin{align*}
	\E_{\{Y_i \}_{i=1}^n} \sup_{g \in \mathcal{G}} (\E - \widehat{\E}_n)[g(Y)]  &\leq \E_{\{Y_i, Y_i' \}_{i=1}^n} \sup_{g \in \mathcal{G}} (\widehat{\E}_n' - \widehat{\E}_n)[g(Y)] \\
	&= \E_{\{ \epsilon_i \}_{i=1}^n}\E_{\{Y_i, Y_i' \}_{i=1}^n} \sup_{g \in \mathcal{G}} \frac{1}{n} \sum_{i=1}^n \epsilon_i (g(Y_i') - g(Y_i)) \\
		&\leq 2 \E_{\{Y_i \}_{i=1}^n} \E_{\{ \epsilon_i \}_{i=1}^n} \sup_{g \in \mathcal{G}} \frac{1}{n} \sum_{i=1}^n \epsilon_i g(Y_i) \;.
	\end{align*}
	Here, the first step exchanges the expectation and supremum, the second step introduces Rademacher variables $\epsilon_i$ to symmetrize the difference, and the last step uses the triangle inequality.
\end{proof}

\end{document}